\documentclass[11pt, a4paper,leqno]{amsart}
\usepackage[utf8]{inputenc}
\usepackage[english]{babel}
\usepackage{mathtools}
\usepackage{amssymb,amscd, amsthm,amsmath,epsfig, graphics,color,latexsym,cancel}   
\usepackage[linktoc=all, pagebackref, hyperindex]{hyperref}%
\hypersetup{colorlinks,  citecolor=blue,  filecolor=blue,  linkcolor=blue,  urlcolor=black}

\usepackage{tikz-cd} 
\usepackage{extpfeil}
\usepackage{tikz}
\usetikzlibrary{matrix,calc}
\usepackage{mathrsfs}
\usepackage[all]{xy}
\usepackage{amsfonts}
\usepackage[normalem]{ulem}
\usepackage{euscript}
\usepackage{amsmath}
\usepackage{enumitem}
\usepackage{ifpdf}
\usepackage{cleveref}
\newcommand{\lar}{\longrightarrow}
\newcommand{\surjects}{\twoheadrightarrow}

\newtheorem{Theorem}{Theorem}[section]
\newtheorem{Lemma}[Theorem]{Lemma}
\newtheorem{Corollary}[Theorem]{Corollary}

\theoremstyle{definition}
\newtheorem{Remark}[Theorem]{Remark}
\newtheorem{Example}[Theorem]{Example}
\newtheorem{Conjecture}[Theorem]{Conjecture}
\newtheorem{Definition}[Theorem]{Definition}

\def\sqr#1#2{{\vcenter{\hrule height.#2pt
			\hbox{\vrule width.#2pt height#1pt \kern#1pt
				\vrule width.#2pt}
			\hrule height.#2pt}}}
\def\phi{\varphi}

\def\VaVa{{\mathcal V}\kern-5pt {\mathcal V}}
\def\gr#1#2{{\rm gr}\, _{#1}(#2)}
\def\gr{{\rm gr}\,}

\def\depth{{\rm depth}\,}

\def\Min{{\rm Min}\,}
\def\codim{{\rm codim}\,}

\def\ker{{\rm ker}\,}
\def\grade{{\rm grade}\,}
\def\rk{\rm rank}

\def\sym#1#2{\mbox{\rm Sym}_{#1}(#2)}

\def\Ext#1#2#3#4{{\rm Ext}\,^{#1}_{#2}({#3},{#4})}

\def\supp#1{{\rm Supp}\, (#1)}

\def\ini{\mbox{\rm in}}

\def\sym{{\mathrm{Sym}}}
\def\cl#1{{\mathcal #1}}

\def\phi{\varphi}

\def\grade{{\rm grade}\,}

\def\QQ{{\bf Q}}

\def\fm{{\mathfrak m}}

\def\NN{\mathbb N}

\def\fp{{\mathfrak p}}

\def\fm{{\mathfrak m}}

\def\NN{\mathbb N}

\def\cl#1{{\cal #1}}
\def\rk{\rm rank}

\newcommand{\excise}[1]{}

\def\NZQ{\mathbb}               
\def\NN{{\NZQ N}}
\def\QQ{{\NZQ Q}}

\def\CC{{\NZQ C}}

\def\G{{\mathcal G}}

\def\opn#1#2{\def#1{\operatorname{#2}}} 
\opn\chara{char} \opn\length{\lambda} \opn\pd{pd} \opn\rk{rk}
\opn\projdim{proj\,dim} \opn\injdim{inj\,dim} \opn\rank{rank}
\opn\depth{depth} \opn\grade{grade} \opn\height{height}
\opn\embdim{emb\,dim} \opn\codim{codim}

\opn\Tr{Tr} \opn\bigrank{big\,rank}
\opn\superheight{superheight}\opn\lcm{lcm}
\opn\trdeg{tr\,deg}
	\opn\reg{reg} \opn\lreg{lreg} \opn\ini{in} \opn\lpd{lpd}
	\opn\size{size} \opn\sdepth{sdepth}
	\opn\link{link}\opn\fdepth{fdepth}\opn\lex{lex}
	\opn\tr{tr}
	\opn\type{type}
	\opn\div{div} \opn\Div{Div} \opn\cl{cl} \opn\Cl{Cl}
	\opn\Spec{Spec} \opn\Supp{Supp} \opn\supp{supp} \opn\Sing{Sing}
	\opn\Ass{Ass} \opn\Min{Min}\opn\Mon{Mon}
	\opn\Ho{H}
	\opn\Ann{Ann} \opn\Rad{Rad} \opn\Soc{Soc}
	\opn\Im{Im} \opn\Ker{Ker} \opn\Coker{Coker} \opn\Am{Am}
	\opn\Hom{Hom} \opn\Tor{Tor} \opn\Ext{Ext} \opn\End{End}
	\opn\Aut{Aut} \opn\id{id}
	
	\opn\nat{nat}
	\opn\pff{pf}
	\opn\Pf{Pf} \opn\GL{GL} \opn\SL{SL} \opn\mod{mod} \opn\ord{ord}
	\opn\Gin{Gin} \opn\Hilb{Hilb}\opn\sort{sort}
	\opn\PF{PF}\opn\Ap{Ap}
	\opn\aff{aff} \opn
	\con{conv} \opn\relint{relint} \opn\st{st}
	\opn\lk{lk} \opn\cn{cn} \opn\core{core} \opn\vol{vol}  \opn\inp{inp} \opn\nilpot{nilpot}
	\opn\link{link} \opn\star{star}\opn\lex{lex}\opn\set{set}
	\opn\width{wd}
	\opn\Fr{F}
	\opn\QF{QF}
	\opn\G{G}
	\opn\type{type}\opn\res{res}
	\opn\log{Log}
	\opn\gr{gr}
	
	\def\pot#1#2{#1[\kern-0.28ex[#2]\kern-0.28ex]}

	\opn\dirlim{\underrightarrow{\lim}}
	\opn\inivlim{\underleftarrow{\lim}}
\begin{document}
		
		\title[Quasihomogeneous isolated singularities]{Quasihomogeneous isolated singularities in terms of syzygies and foliations}
		
		\author{Hamid  Hassanzadeh}
		\address{Departamento de Matemática, Centro de Tecnologia,  Universidade Federal do Rio de Janeiro, 21941-909 Rio de Janeiro, RJ, Brazil}
		\email{hamid@im.ufrj.br}
		\thanks{The first author was partially supported by a grant from CAPES (Brazil) (Finance Code 001), FAPERJ (2025, APQ1), and CNPq (302761/2026-8).}

		\author{Abbas Nasrollah Nejad}
		\address{Department of Mathematics, Institute for Advanced Studies in Basic Sciences (IASBS), Zanjan 45137-66731, Iran \newline \indent
        Instituto de Ciência e Tecnologia, Universidade Federal de S\~ao Paulo (UNIFESP),
        S\~ao Jos\'e dos Campos, SP, Brazil
        }
		\email{annejad@unifesp.br}

\author{Aron Simis} 
\address{Departamento de Matemática, Universidade Federal da Pernambuco, Recife, Pernambuco, 50740-560, Brazil,}
\email{aron.simis@ufpe.br}
\thanks{The third author was partially supported by a grant from CNPq (Brazil) (304800/2024-4).}

		

\begin{abstract}
			One considers quasihomogeneous isolated singularities of hypersurfaces in arbitrary dimensions through the lenses of three apparently quite apart themes: syzygies, singularity invariants, and foliations. In the first of these, one adds to the well-known result of Saito's  a syzygy-theoretic characterization of a quasihomogeneous singularity affording an effective computational criterion. 
			In the second theme,  one explores the Milnor–Tjurina difference number from a commutative algebra viewpoint. Building on the Brian\c{c}on–Skoda theorem and exponent, we extend previously known inequalities by Dimca and Greuel to arbitrary dimension and provide algebraic formulas involving the syzygy-theoretic part and reduction exponents. 			In the last theme one recovers and brings up  an algebraic light a result of Camacho and Movasati by establishing a couple of characterizations of quasihomogeneous isolated singularities in terms of the generators of the module of invariant vector fields. 
		\end{abstract}
 		\maketitle

\section{Introduction}
		
	Given an integer $n\geq 2$, let $R$ denote  the ring $\pot{k}{x_1,\ldots,x_n}$ of formal power series over a field $k$ of characteristic zero.
	 Let $f\in R$ define an isolated hypersurface singularity at the null point.
	 The main drive of this work is to elucidate further aspects of quasihomogeneous isolated singularities by bringing out some of their applications as related to the Milnor--Tjurina number difference and to logarithmic derivations.

In order to describe the contents of the sections one recalls some of the main notions employed in the paper, along with a recollection of the related literature.
Assume throughout that $f\in R=\pot{k}{x_1,\ldots,x_n}$ defines an isolated hypersurface singularity at the null point.

First and foremost, the gradient ideal $J_f:=\langle\partial f/\partial x_1,\ldots,\partial f/\partial x_n\rangle$ and the Jacobian ideal $I_f:=\langle f,J_f\rangle$ -- quite often referred to respectively as the Jacobian ideal and the Tjurina ideal (\cite[Definition 2.1]{GLS}), with respective {\em Milnor algebra} $M_f:=R/J_f$ and {\em Tjurina algebra} $T_f:=R/I_f$.
Since $f$ defines an isolated hypersurface singularity at the null point, then both ideals are primary to the maximal ideal generated by $\{x_1,\ldots,x_n\}$.

 The first of these ideals takes us naturally to the nonnegative integer
\begin{equation*}
    \mu_f = \dim_{k} R/ {J_f}. 
\end{equation*}

This integer  is called the {\it Milnor number} of the isolated singularity defined by $f$ because of its interpretation as (or tight relationship to) the designation $\mu$ given by Milnor (\cite[Chapter 7]{Milnor}) to a number measuring the amount of {\it degeneracy} at an isolated critical point $p$ of a polynomial $f$ -- where $p$ is said to be {\it non-degenerate} if the Hessian matrix of $f$ at $p$ is non-singular -- or yet, as  the multiplicity of $p$ as a ``solution''  of the polynomial equations $\partial f/\partial x_1= \cdots =\partial f/\partial x_n=0$.
 In the case where $k$ is the complex field, $\mu_f$ has a strong topological facet in terms of the Milnor fibration, implying for instance that each fiber has the homotopy type of a bouquet (wedge sum) of  $\mu_f $ spheres of dimension $ n-1 $.

The Jacobian ideal above in turn provides us with the nonnegative integer
\begin{equation*}
    \tau_f = \dim_{k} R/I_f,
\end{equation*}
 measuring the complexity of the isolated singularity from an algebraic/deformation theoretic viewpoint.
It is called the {\it Tjurina number} of the singularity, a designation given by Gert-Martin Greuel as a homage to Galina Tjurina (deceased).
In a certain sense, it measures the dimension of the space of first-order deformations of the singularity.

Clearly, $\tau_f \leq \mu_f $ as a consequence of the natural ideal theoretic inclusion $J_f\subset I_f$.
Thus, one is naturally led to wonder about the obstruction to the equality $\tau_f =\mu_f $.
The problem of understanding the difference $\mu_f-\tau_f$ has appealed to several authors (see, e.g., \cite{Almiron}, \cite{BayerHefez}, \cite{Liu},  \cite{Watari}), with varied partial results.

	As observed by Teissier quite early  in~\cite{Teissier}, the notion of integral closure of an ideal plays a significant role in singularity theory. Recall that the \textit{integral closure} of an ideal $I$ in a ring $R$ is the set $\overline{I}$ of elements $f\in R$ satisfying an equation $f^m+a_{1}f^{m-1}+\ldots+a_{m-1}f+a_m=0$, for certain $a_i\in I^i$, $i=1,\ldots,m$. This set is an overideal of $I$ and  the latter is said to be \textit{integrally closed} in $R$ if  $I = \overline{I}$. A related notion in this context is that of a {\em reduction} $J \subseteq I$ in the sense that $J$ is a subideal satisfying the equality $JI^t =I^{t+1}$, for a sufficiently large $t$. The \textit{reduction number} of $I$ relative to $J$ is the integer
	\[{r}_J(I)= \min \{t \colon \; JI^s=I^{s+1}, \text{ for all } s \geq t \}.\]
	 These two concepts relate to each other in the sense that $f\in\overline{I}$ if and only if $I$ is a reduction of the ideal $(f,I)$. For further related properties, the reader is  referred~\cite{HunekeSwanson} or \cite[Subsection 7.3.4]{SimisBook}.
	
	The following basic result shows that the two ideals $J_f$ and $I_f$ are tightly related.
	
	\begin{Theorem}\label{Tessier} 
		Let $R=\pot{k}{x_1,\ldots,x_n}$ where $k$ is a field of characteristic zero. If $f\in R$  is non-invertible, then it belongs to the integral closure of the ideal $\langle x_1\partial f/\partial x_1,\ldots, x_n\partial f/\partial x_n\rangle$. In particular, $J_f$ is a reduction of $I_f$.  
	\end{Theorem}	
A proof of this theorem drawing upon the valuation criteria of integral dependence is provided in \cite[7.1.5]{HunekeSwanson}. This result was already known to Lejeune-Jalabert and Teissier (\cite {L-JT74}) in the context of the ring of convergent power series, and to Scheja--Storch (\cite{SchejaStorch}) in the  weaker form that $f$ belongs to $\overline{J_f}$ (but sufficient for the applications in this paper), in the context of formal power series.	

Next, a quick deflection to Hilbert multiplicity.
For a Noetherian local ring  $(A,\fm)$  of dimension $d$ with an infinite residue field $A/\fm$, the \textit{Hilbert--Samuel function} of an $\fm$-primary ideal $I \subset A$ is defined as $\mathrm{HF}_I(n): = \lambda(A/I^n)$,  $n\in \NN$. For sufficiently large $n$, this function agrees with a polynomial of degree $d$ given by

\begin{equation}\label{normalized}
	{P}_I(t)=e_0(I)\binom{t+d-1}{d}-e_1(I)\binom{t+d-2}{d-1}+\cdots+(-1)^de_d(I),
\end{equation}
where $e_i(I)\in \mathbb{Z}$ are called the \textit{Hilbert Coefficients} of $A$ with respect to $I$. The leading coefficient $e_0(I)$ is the  \textit{multiplicity} of $I$, commonly denoted by $e(I,A)$.  
When $J\subset I$ is a reduction of $I$, one has $e(I,A)=e(J,A)$ \cite[Theorem 14.13]{Matsumura}.  Furthermore, if  $A$ is a Cohen-Macaulay ring and $J$ is generated by a system of parameters then $\lambda(A/J)=e(J,A)$ \cite[Theorem 17.11]{Matsumura}.   

Northcott \cite{Northcott} showed that, for an $\fm$-primary ideal in a Cohen-Macaulay local ring $(A,\fm)$,  the following holds
$$e_1(I)\geq e_0(I)-\lambda(A/I)\geq 0. $$	

In \cite[Theorem 2.1]{Huneke2} Huneke determined when the equality on the left holds. Namely, he has shown that if $e_1(I)= e_0(I)-\lambda(A/I)$ then $e_2=\cdots=e_d=0$,  $\mathrm{HF}_I(t) =P_I(t)$ for all $t\geq 1$ and for any minimal reduction $J\subset I$, $I^2=JI$; the converse -- namely, to the effect that,  if $I$ admits a minimal reduction  $J\subset I$ with $I^2=JI$ then  $e_1(I)= e_0(I)-\lambda(A/I)$ -- had been obtained earlier by Kubota \cite{Kubota}.

Finally, quite generally, the module of {\em logarithmic derivations} associated to an element $f\in S=k[x_1,\ldots,x_n]$ is defined as the $S$-module $\mathrm{Derlog}_S(f):=\{\eta\in {\rm Der}_k(S)\,|\, \eta (f)\in \langle f \rangle\},$
where ${\rm Der}_k(S)$ stands for the (free) module of derivations of $S$ over $k$.
If $S$ is standard graded,  $f\in S$ is a homogeneous polynomial, and  if ${\rm char}(k)$ does not divide $\mathrm{deg}(f)$ (e.g., if $k$ has characteristic zero),  then  this module has  a splitting structure
\begin{equation}\label{splitting}
	\mathrm{Derlog}_S(f)= \mathrm{Syz}(J_f)\oplus S\theta_E,
\end{equation}
where $\mathrm{Syz}(J_f)\subset S^n$ denotes the module of first syzygies of the partial derivatives of $f$ and $J_f\subset S$ denotes the Jacobian ideal of $f$, with $\theta_E$ standing for  the Euler derivation. This splitting is well known (see, e.g., \cite[Proposition 4.3.8]{SimisBook}).
If $f$ is not homogeneous, a different approach may be required.
In this realm we are interested in the residue module
\begin{equation}\label{Derlog_residue}
\mathcal{E}_f: = \mathrm{Derlog}_S(f) / \text{KD},
\end{equation} 
with $\text{KD}$ standing for the submodule generated by the 
\textit{Koszul derivations}
$$\left\{ f_i \frac{\partial}{\partial x_j} - f_j \frac{\partial}{\partial x_i} \mid 1 \leq i, j \leq n \right\}  
\bigcup \left\{ f \frac{\partial}{\partial x_i} \mid 1 \leq i \leq n \right\},$$ 
where $f_i = \frac{\partial f}{\partial x_i}$.  
This residue module may be called the \textit{module of essential derivations}, a terminology used in \cite{DimcaSticlaru} (see also \cite{Miranda} in a slightly different context).

So much for general nonsense.  One next briefly describe the contents of the sections. 
	
Section~\ref{quasihomogeneous-all} delves into the property of quasihomogeneity, where the main result is Theorem~\ref{QuasiHomogeneous}. It collects  known results around Saito's theorem Theorem~\ref{Saito_original}. Among the
equivalent conditions for quasihomogeneity in Theorem~\ref{QuasiHomogeneous}, we recover the syzygy-matrix rank characterization established in~\cite{ABDM} for projective hypersurfaces with isolated singularities.
Our contribution is a simple and direct algebraic proof in the formal local setting. Using Fitting ideals, we relate the rank of the syzygy matrix to the minimal number of generators of $I_f$.
Then, using Theorem~\ref{Tessier} and multiplicity  theory, we prove that if $\nu(I_f)=n$, then $f\in J_f$. The present proof is commutative algebra minded.

Our approach also provides an alternative proof of a previous result obtained through Zariski's structural analysis of the module of derivations for an isolated singularity in a complete ring (see Martsinkovsky~\cite[Proposition 1.1]{Marts}). 

A close discussion to parts of this section has been taken up by Dimca and Sticlaru (\cite{DimcaSticlaru}). They start with a  {\em homogeneous} polynomial $f$, then pass to an affine chart by assuming that the hyperplane $x_0=0$ is transversal to $V(f)$. Here, the Euler relation for $f$ is essential in their argument, providing in particular a built-in homogeneity for the singularities (see \cite[Eq. (2.1)]{DimcaSticlaru}). In contrast, the present treatment does not assume homogeneity and is formal power series minded, where the Euler relation is unavailable.
Here, instead one appeals to Teissier’s criterion (and its consequences) to establish the necessary homogeneity properties of singularities. One hopes that this overcoming procedure has  independent interest in singularity theory.

The content of Section~\ref{Difference} concerns more recent results.
Drawing upon the Brian\c{c}on--Skoda theorem, Y. Liu~\cite{Liu} proved the inequality $\mu_f\leq n\tau_f$.
In the case of a plane curve singularity ($n=2$), Dimca and Greuel~\cite[Theorem 1.1(3)]{DG} improved this bound, showing that $\mu_f<2\tau_f$.
They further conjectured that the sharper inequality $\mu_f<\frac{4}{3}\tau_f$ should hold.
This conjecture has been proved in~\cite{Almiron}. 

More recently even, Ma and Zuo~\cite{MaZou} studied the Milnor--Tjurina quotient using Hilbert--Samuel, Hilbert--Kunz, and $s$-multiplicities, obtaining sharp upper bounds in arbitrary dimension.
In characteristic zero, their results give, in particular, $\mu_f/\tau_f\leq 4/3$ for $n=2$ and $\mu_f/\tau_f\leq 3/2$ for $n=3$.

Here, we envisage the Milnor--Tjurina difference number from a commutative algebra point of view.
In Theorem~\ref{bounds}(ii) we extend the above result by Dimca--Greuel to arbitrary dimension $n\geq 2$.
That is, we prove the inequality $\mu_f<n\tau_f$, and even stronger relations by applying the Brian\c{c}on--Skoda exponent.
Corollary~\ref{smalldifference} provides reasonable evidence to the belief that a smaller Milnor--Tjurina difference number corresponds to a lower Brian\c{c}on--Skoda exponent.
Theorem~\ref{DiferMT} establishes the relation
$$
\mu_f-\tau_f=\lambda(I_f^2/J_fI_f)+\lambda(\delta(I_f)),
$$
where $\delta(I_f)$ is the syzygetic module $\ker(\sym_R^2(I_f)\surjects I_f^2)$ introduced by Vasconcelos and the third author~\cite{Aron-Wolmer}.

Alongside these results, one introduces Conjecture~\ref{conject}, which carries a couple of suggestive consequences if proved to hold.
It extends a previously stated version which has recently been claimed to be true in \cite{ZhangZuo}, whereas a more encompassing version of the previous conjecture as proposed in \cite{MaZou} has been proved  to fail in same \cite{ZhangZuo}.
It would seem like there is as yet quite a bit of endeavor to understand the interrelation among these various versions.
					  
 In Section~\ref{logarithmic}, we tie the module introduced in (\ref{Derlog_residue}) to the theory of algebraic foliations of plane algebraic curves in characteristic zero.
 The study of such foliations has long captured the attention of considerable authors  since the pioneering work of Poincar\'e and Dulac. In this paper we are basically interested in a more algebraic--geometric approach leading to the vertent of vector fields. The recent work of Camacho, Movasati, and Hertling~\cite{Camacho_et_al} falls along this line, raising the opportunity to rounding up some of their main results according to the present viewpoint. Now, given a polynomial $f \in k[x,y]$ over a field $k$ of characteristic zero, these authors introduce the module  
 \[
 E_f: \,=\, \{ \,\omega \in \Omega^1_{\mathbb{A}^2_k} \;\mid\; df \wedge \omega = f\alpha \text{ for some } \alpha \in \Omega^2_{\mathbb{A}^2_k}\,\},
 \]  
 which encodes all algebraic foliations in $\mathbb{A}^2_k$ leaving the curve $f=0$ invariant.
 One of their main results is as follows.
 \begin{Theorem}\label{CamachoHossein}\cite[Theorem 1]{Camacho_et_al} Assume that $\mu_f< \infty$. If all the singularities of $f = 0$ are quasihomogeneous then there exists $\omega_{f}\in E_{f}$ such that $fdx,fdy,df,\omega_{f}$ generate the $k[x,y]$-module $E_{f}$.
 \end{Theorem}
 
 In the present paper, one brings out the previously mentioned Koszul derivations, which as usual in the presence of vector fields are thought of as ``trivial derivations''.
 Rewriting $E_f$ in terms of derivations yields back $ \mathrm{Derlog}_S(f)$, and the present first result shows that its minimal number of generators modulo the Koszul derivations 
 is encoded in the first Koszul homology module of $I_f$.  Theorem~\ref{E_fprincipal} and Corollary~\ref{Corolaryfinal} in this work extend Theorem~\ref{CamachoHossein}  to arbitrary dimension for hypersurfaces with isolated singularities.
 In addition, one proves that Theorem~\ref{CamachoHossein}  is actually an if and only if statement, thus rounding up the complete picture.
 In particular, it further clarifies the assertion in \cite{Camacho_et_al} to the effect that the proof of Theorem~\ref{CamachoHossein} only works for curves with quasihomogeneous singularities.

 Once again the results of this section may be confronted with the ones in \cite{DimcaSticlaru}, where the standing hypothesis has $f$ homogeneous and $x_0=0$ transversal to $V(f)$. The latter is not a mild technicality as homogenizing an affine polynomial $f\in k[x_1,\ldots,x_n]$ with an isolated singularity at the origin need not produce a transverse hyperplane as such, and so their route cannot be applied blindly in the non-homogeneous (or formal) context as considered in this work.
 
 For the reader's convenience, the main results of this paper are brought in the following theorems: Theorem~\ref{QuasiHomogeneous}, Theorem~\ref{DiferMT}, Theorem~\ref{bounds}, and Theorem~\ref{E_fprincipal}. 

\section{An encore on quasi-homogeneous isolated hypersurface singularities}\label{quasihomogeneous-all}
One recalls the landmark  result of K. Saito concerning the distinguished class of hypersurface singularities called  quasihomogeneous. 

\begin{Definition}
A polynomial $f\in k[x_1,\ldots,x_n]$ is called \textit{quasihomogeneous} (or \textit{weighted homogeneous}) if there exist  rational numbers $0< r_i \, (1\leq i\leq n)$ such that $f$ is a $k$-linear combination of monomials $x_1^{m_1}\cdots x_n^{m_n}$, where      $\sum_{i=1}^n r_im_i=1$.
\end{Definition}


Saito proved the following result:
\begin{Theorem}\label{Saito_original}
	\cite[Satz]{Saito} Let $f \in \widetilde{R}:=\mathbb{C}\left\{x_1,\ldots, x_n\right\}$ denote a convergent power series having an isolated singularity at the null point. 
The following are equivalent$:$
\begin{enumerate}
 \item[{\rm (i)}]  $\widetilde{R}$ admits an holomorphic coordinate change transforming $f$ in a quasihomogeneous polynomial. 
 \item[{\rm (ii)}]  $ f\in J_f$, where $J_f$ is the gradient ideal of $f$ in $\widetilde{R}$.
\end{enumerate}
\end{Theorem}

\begin{Remark}
 Saito's argument throughout is formal -- i.e., over $\mathbb{C}[[x_1,\ldots,x_n]]$ -- though he exerts some effort to show that the implication (ii) $\Rightarrow$ (i) works for convergent power series as well.
\end{Remark}
Let now $k$ stand for a field of zero characteristic, and let $R:=\pot{k}{x_1,\ldots,x_n}$ denote the ring of formal power series over $k$ (i.e., the $\langle x_1,\ldots,x_n \rangle_{\langle x_1,\ldots,x_n\rangle}$-adic  completion of the local ring $k[x_1,\ldots,x_n]_{\langle x_1,\ldots,x_n\rangle})$. 

Through this section, $R/\langle f\rangle$ stands for a hypersurface ring with an isolated singularity (always assumed to be so at the null point, unless otherwise stated).

\begin{Definition}\label{quasihomogeneous_redefined} 
A hypersurface ring $R/\langle f\rangle$ as above
will be called {\rm quasihomogeneous} if $f\in J_f$ (as in Theorem 2.2 (ii)).
In this case, if no confusion arises, 
 $f$ will often be said quasihomogeneous as well, although it is not necessarily polynomial.
\end{Definition}
Since $R$ is a Noetherian local ring, if $R/\langle f\rangle$ is a hypersurface ring with an isolated singularity which is quasihomogeneous, then the partial derivatives of $f$ form a regular sequence in $R$. Hence, the Milnor algebra $M_f$ is a complete intersection, in particular an Artinian Gorenstein ring.

Using a result of Zariski on the structure of the module of derivations of an isolated singularity in a complete ring, Martsinkovsky~\cite[Proposition 1.1]{Marts} proved that the Tjurina algebra $T_f$ is Gorenstein if and only if $R/\langle f\rangle$ is quasihomogeneous.

One provides below a different proof of this result, applying Theorem \ref{Tessier} and adding further details.  In what follows, let $\mathcal{Z}(I_f)$ stand for the defining matrix of a finite free presentation of $I_f$ based on the generators  $\{f_{x_1},\ldots, f_{x_n},f\}$ (the {\em syzygy matrix} of the given generators).
Let $\mathcal{Z}(I_f)_{\bf 0}$ denote the matrix resulting from $\mathcal{Z}(I_f)$ upon evaluation $x_i\mapsto 0$, $\forall$ $i$. 

Here and throughout the rest of the paper,  the symbol $\nu(I)$ will stand for the minimal number of generators of an ideal $I$ of a Noetherian local ring, and $\lambda_R(-)$ will denote the length function.

\begin{Theorem}\label{QuasiHomogeneous} 
	Let $R/\langle f\rangle$ be a  hypersurface ring with an isolated singularity, with  $f\in R=\pot{k}{x_1,\ldots,x_n}$.  The following are equivalent$:$
	\begin{enumerate}
		\item[{\rm (a)}] $R/\langle f\rangle$  is quasihomogeneous.
		\item[{\rm (b)}]  The Tjurina algebra $T_f$ is an Artinian Gorenstein ring. 
        \item[{\rm (c)}] The minimal number of generators of $I_f$ is $n$.
        \item[{\rm (d)}] $\mathcal{Z}(I_f)_{\bf 0}$  has positive rank. 
	\end{enumerate}
\end{Theorem}
\begin{proof}

$(a)\Rightarrow (b)$. This is clear from the definitions, as in this case the Tjurina algebra and the Milnor algebra are one and the same.

$(b)\Rightarrow (c)$. This is also clear as a strict  almost complete intersection  is not a Gorenstein ring (\cite[Corollary 1.2]{Kunz}).

\smallskip

$(c) \Leftrightarrow (d)$  
Consider the presentation matrix of $I_f$,  based on the generators   $\{f_{x_1},\ldots, f_{x_n},f\}$
\[
R^m \xrightarrow{\mathcal{Z}(I_f)} R^{n+1} \longrightarrow I_f \longrightarrow 0.
\]  
Since $R/\langle f\rangle$ is a hypersurface ring with an isolated singularity, $\codim(I_f)\geq n$; in particular, the minimal number of generators of $I_f$ is  at least $n$.   By general properties of Fitting ideals (\cite[Proposition 20.6]{EisenbudBook}), the minimal number of generators of $I_f$ is at most $n$ if and only if the maximal ideal $\langle x_1,\ldots,x_n\rangle \subset R$ does not contain $\operatorname{Fitt}_n(I_f) = I_1(\mathcal{Z}(I_f))$.
Therefore, as $R$ is local, the minimal number of generators of $I_f$ is exactly $n$ if and only if $I_1(\mathcal{Z}(I_f))$ is the unit ideal, i.e.,  if and only if at least one among the entries of $\mathcal{Z}(I_f)$ is not contained in the  ideal $\langle x_1,\ldots,x_n\rangle \subset R$.

\smallskip

$(c) \Rightarrow (a)$  By the standing data, $R$ is a Cohen-Macaulay local ring of dimension $n$ and $J_f$ is generated by a regular sequence forming a system of parameters of $R$. Thus, $\lambda_R(R/J_f) = e(J_f, R)$ (see \cite[17.11]{Matsumura}), where $e(J_f, R)$ is the multiplicity of $R$ with respect to $J_f$.

Similarly, by the assumption in (c), $\nu(I_f)=n$,  while $\codim(I_f)\geq \codim(J_f)=n$. Again, since $R$ is Cohen--Macaulay, $I_f$ can be generated by a regular sequence that  forms a system of parameters of $R$. Thus,   
$\lambda_R(R/I_f) = e(I_f, R)$ as well.

By Theorem~\autoref{Tessier}, $J_f$ is a reduction of $I_f$, which means that $e(J_f, R) = e(I_f, R)$ \cite[14.13]{Matsumura}. Consequently,  
\[
\lambda_R(R/J_f) = \lambda_R(R/I_f).
\]  
From the exact sequence  of modules of finite length
\[
0 \longrightarrow R/\langle J_f:f\rangle \longrightarrow R/J_f \longrightarrow R/I_f \longrightarrow 0
\]  
follows   
$\lambda_R(R/J_f:f) = 0,$ i.e., 
 $R = J_f:f$, hence $f\in J_f$ as desired.  
\end{proof}

As a consequence, one has a lower bound for $\tau_f$.
 
\begin{Corollary}\label{leastTau}  Let $R/\langle f\rangle$ be a  hypersurface ring with an isolated singularity, with  $f\in R=\pot{k}{x_1,\ldots,x_n}$. If $f$ is not quasi-homogeneous then $\tau_f\geq 3$.
\end{Corollary}
\begin{proof} Set $\fm:=\langle x_1,\ldots,x_n\rangle$, and assume that $\tau_f\leq 2$. Since the Tjurina algebra $T_f$ is Artinian local, then $\tau_f=\lambda(T_f)=\lambda(\fm)+\lambda(T_f/\fm)=\lambda(\fm^2)+\lambda(\fm/\fm^2)+ \lambda(T_f/\fm)$, thus implying  that $\fm^2=0$ and that $\fm$ is generated by at most $1$ element.  
Therefore,  $T_f$ is an Artinian local ring with cyclic socle, so $T_f$ is Gorenstein. By Theorem \autoref{QuasiHomogeneous},  $f$ must be quasi-homogeneous, contradicting the assumption.
\end{proof} 

\begin{Remark}\rm
(1)  It should be noted that both the main argument in the above implication (c) $\Rightarrow$ (a), and the one in the implication  (1) $\Rightarrow$ (2)  in Martsinkovsky's,  contradict one and the same assumption but follow different paths.

(2) Besides the algebraic significance of Theorem~\ref{QuasiHomogeneous},  the implication \((d) \Rightarrow (a)\) provides a particularly simple and practical criterion of quasihomogeneity. Furthermore, the ideal generated by the entries of the syzygy matrix is the entire ring \( R \), a result  recently brought up 
in~\cite{ABDM}.

(3) Concerning the computational issue of testing quasihomogeneity, it is often more practical to work with a different set of generators for the ideal $I_f$, such as a Gr\"obner basis when the base ring is a polynomial ring. In this case, the presentation matrix of $\mathcal{Z}_f$ is not necessarily an $(n+1)$-row matrix. Consequently, one must adapt the criterion in part (d) as follows: the evaluated matrix $\mathcal{Z}(I_f)_{\bf 0}$ should have {\it corank} at most $n$. In other words, if the ideal \( I_f \) has a presentation with \( N \) rows, then the rank of $\mathcal{Z}(I_f)_{\bf 0}$ must be at least \( N - n \). The package {\bf FastMinors} in Macaulay2 provides the function {\bf isRankAtLeast} to implement this test  efficiently.
\end{Remark} 

\begin{Example}
(i) Let $f=xy+xz+yz+xyz\in k[x,y,z]$. Then $f$, considered as an element of $R=k[[x,y,z]]$,  defines an isolated surface singularity at the null point. The transpose of 
	\[\begin{bmatrix}
		xz+2x+z& -(yz+z)& z^2+z&-(z+2)	
	\end{bmatrix} 
	\]
	 is a syzygy element of the ideal $I_f=\langle f_x,f_y,f_z,f\rangle \subset R$ that does not vanish at the null point.
	 Thus, $R/\langle f\rangle$ is quasihomogeneous.  
	 
(ii)   \cite[Example 2.6]{Abbas}  Consider the curve singularity defined by  $f=x^4+x^3y^2+y^6\in k[x,y]$. A computation with \cite{DGPS} yields the following syzygy matrix  of $I_f=\langle f,f_x,f_y\rangle$:
	\[ 
\mathcal{Z}(I_f)=	\left(\begin{matrix}
		24x+18y^2& 96y+18xy\\
		-6x^2-4xy^2&-24xy-4x^2y+2y^3\\
		-4xy-3y^3& -x^2-16y^2-3xy^2
		\end{matrix}\right).\] 
	Then $\rk (\mathcal{Z}(I_f)(0,0))=0$ which implies that $R/\langle f\rangle$ is not quasihomogeneous. 
\end{Example}


\section{Complements on the Milnor--Tjurina  difference number}\label{Difference}
As previously, let $R/\langle f \rangle$ be a hypersurface with an isolated singularity, where $f \in R = \pot{k}{x_1,\ldots,x_n}$.

\subsection{Reduction exponent one, and the syzygy-theoretic part}

Driving upon the preliminares in the Introduction, Theorem~\ref{Tessier} implies that for an isolated singularity  $f \in A:=R= \pot{k}{x_1,\ldots,x_n}$, one has
$$\mu_f-\tau_f=\lambda(R/J_f)-\lambda(R/I_f)=e(J_f,R)-\lambda(R/I_f)=e(I_f,R)-\lambda(R/I_f).$$

The main drive in this part is the case where the reduction exponent of $I_f$ with respect to $J_f$ is $1$.
A necessary and sufficient condition has long been stated by Huneke \cite{Huneke2} and Kubota \cite{Kubota} to the effect that $e_1(I_f,R)=\mu_f-\tau_f$. 


One presently brings up a new algebraic formula for $\mu_f-\tau_f$ by bringing in the invariant $\delta(I_f)$ introduced in~\cite{AronSyzgetic}. The simplest way to introduce $\delta(I_f)$ is as the kernel of   the natural surjection $\sym_R^2(I_f)\surjects I_f^2$ of the symmetric algebra of $I_f$ onto its Rees algebra as regarded in degree $2$ (\cite[Corollary 1.2]{Aron-Wolmer}).


\begin{Theorem}\label{DiferMT}
	Let $R/\langle f \rangle$ be a hypersurface with an isolated  singularity, with $f\in R=\pot{k}{x_1,\ldots,x_n}$. 
	Then$:$
	\begin{enumerate}
		\item[{\rm (i)}]
		$\mu_f-\tau_f=\lambda(I_f^2/J_fI_f)+\lambda(\delta(I_f))$
		
		\noindent		In particular, $\mu_f-\tau_f=\lambda(\delta(I_f))$ if and only if the reduction number of $I_f$ relative to $J_f$  is at most $1$.
		\item[{\rm (ii)}] The following are equivalent:
		\begin{enumerate}
			\item[{\rm (a)}] $R/\langle f \rangle$ is quasihomogeneous.
			\item[{\rm (b)}] $\delta(I_f)=0$.
		\end{enumerate}
			\item[{\rm (iii)}] If $\mu_f-\tau_f=1$ then $I_f^2=J_fI_f$.
	\end{enumerate}
	
\end{Theorem}
\begin{proof}
	As seen in the previous section, 
$R/\langle f \rangle$ is quasihomogeneous if and only if $\mu_f=\tau_f$. 
	
	\smallskip
	
	\noindent (i) Consider the exact sequence:
	\begin{equation} \label{nonsense-seq}
		0 \longrightarrow J_f/I_f J_f \longrightarrow R/J_f I_f \longrightarrow R/J_f \longrightarrow 0,
	\end{equation}
	where all terms have finite length. Since $J_f$ is generated by an $R$-regular sequence, one has:
	\[
	(J_f/J_f^2) \otimes_R (R/I_f) \cong (R/I_f)^n.
	\]
	Clearly,
	\[
	J_f/J_f I_f \cong (J_f/J_f^2) \otimes_R (R/I_f).
	\]
	Taking lengths in (\ref{nonsense-seq}) gives
	\begin{equation} \label{1}
		\lambda(R/J_f I_f) = \lambda(R/J_f) + \lambda(J_f/J_f I_f) = \mu_f + n \tau_f.
	\end{equation}
	Using this and the exact sequence
	\[
	0 \longrightarrow I_f^2/J_f I_f \longrightarrow I_f/J_f I_f \longrightarrow I_f/I_f^2 \longrightarrow 0,
	\]
	it obtains
	\begin{equation} \label{reduction2}
		\lambda(I_f^2/J_f I_f) + \lambda(I_f/I_f^2) = \lambda(I_f/J_f I_f) = \mu_f + (n-1) \tau_f.
	\end{equation}
	Now, consider a presentation of $I_f$:
	\[
	0 \longrightarrow \mathcal{Z}(I_f) \longrightarrow R^{n+1} \longrightarrow I_f \longrightarrow 0.
	\]
	Tensoring with $R/I_f$ gives an exact sequence \cite[Section 1, Proposition 2.4]{Aron-Wolmer}:
	\begin{equation} \label{syzygetic}
		0 \longrightarrow \delta(I_f) \longrightarrow H_1 \longrightarrow (R/I_f)^{n+1} \longrightarrow I_f/I_f^2 \longrightarrow 0
	\end{equation}
	of modules of finite length.
	Here, $H_1$ is the first Koszul homology module of $I_f$. It is the canonical module of $R/I_f$ \cite[1.6.16]{Herzog.B}, and moreover,
	\[
	\lambda(H_1) = \lambda(R/I_f) \quad \text{\cite[3.2.12]{Herzog.B}}.
	\]
	Thus, we get:
	\begin{equation} \label{conormal}
		\lambda(I_f/I_f^2) = n \tau_f + \lambda(\delta(I_f)).
	\end{equation}
	From \eqref{reduction2} and \eqref{conormal}, we derive the equality
	\[
	\mu_f - \tau_f = \lambda(I_f^2/J_f I_f) + \lambda(\delta(I_f)),
	\]
	as desired.
	
	\medskip
	(ii) The implication (a) $\Rightarrow$ (b) is clear, since the right-hand side of the formula in (i) is a sum of nonnegative integers. Alternatively, and more directly, (a) gives that $I_f$ is generated by a regular sequence, in which case it is clearly a syzygetic ideal.
	
	(b) $\Rightarrow$ (a): If $\delta(I_f) = 0$, then from \eqref{syzygetic} we obtain the short exact sequence:
	\[
	0 \longrightarrow H_1 \longrightarrow (R/I_f)^{n+1} \longrightarrow I_f/I_f^2 \longrightarrow 0.
	\]
	Since $H_1$ is the canonical module of an Artinian ring, it is an injective module. Thus, this sequence splits:
	\[
	(R/I_f)^{n+1} \simeq H_1 \oplus I_f/I_f^2.
	\]
	Tensoring with $k$ gives the following number equality:
	\[
	\nu(H_1) + \nu(I_f/I_f^2) = n+1.
	\]
	 Since $\nu(I_f/I_f^2) = \nu(I_f)$ and $\nu(H_1) \neq 0$, we get $\nu(I_f) = n$.
	Then, by Theorem~\ref{QuasiHomogeneous}, we conclude that $I_f = J_f$, as claimed.

	 (iii) Assume that $\mu_f-\tau_f=1$. Then $f$ is not quasihomogeneous, so by part (ii), $\lambda(\delta(I_f))\geq 1$. Hence, part (i) gives 
$1=\lambda(I_f^2/J_fI_f)+\lambda(\delta(I_f))$
which forces $\lambda(I_f^2/J_fI_f)=0$. Therefore
$I_f^2=J_fI_f$, as desired.
\end{proof}

\subsection{Variations on the Brian\c{c}on-Skoda exponent}

A now sufficiently standard procedure to compare Milnor and Tjurina numbers is to look at the ratio $\mu_f/\tau_f$. 
In order to approach it over here, we recall the following well-known notion introduced in \cite{BS}.

\begin{Definition}\label{BSdefinition}
Let $R/\langle f \rangle$ define an isolated hypersurface singularity, with $f \in R = \pot{k}{x_1,\ldots,x_n}$. The {\it Brian\c{c}on-Skoda exponent} of $R/\langle f \rangle$ is the integer 
$$e^{\mathrm{BS}}(f) := \min \{ d \in \mathbb{N} \mid f^d \in J_f \}.$$
\end{Definition}
From \cite{BS}, one has $e^{\mathrm{BS}}(f) \leq n$.  
\begin{Remark}\label{homog_BSexponent}\rm Note that, from Definition~\ref{quasihomogeneous_redefined}, $e^{\mathrm{BS}}(f)=1$ if and only if $R/\langle f \rangle$ is quasihomogeneous. Therefore, minimal possible Brian\c{c}on-Skoda exponent means quasihomogeneity, i.e., this exponent measures how far out is $R/\langle f \rangle$ from being quasihomogeneous.
\end{Remark} 

No wonder, one expects a strong intertwining among the Brian\c{c}on-Skoda~\cite{BS} exponent, the Tjurina number, and the Milnor number. 
The following linear algebra content gives a background. 
Namely, for an isolated singularity $f \in R = k[[x_1,\ldots,x_n]]$ over a field of characteristic zero, the algebra $R/J_f$ is an Artinian local ring, in particular,  a finite $k$-vector space. As in \cite[Section 2]{Iarobbino-et-al}  and \cite[Section 4]{Camacho_et_al}, one considers the $k$-linear transformation $\mathfrak{L}: R/J_f \to R/J_f$ given by multiplication by $f$. If char$(k)=0$,  $\mathfrak{L}$ is a nilpotent map and thus has a Jordan form of type, say, $\lambda_1 \geq \cdots \geq \lambda_r$ for suitable $r$. 
Introducing a new variable $t$ over $k$, the algebra $R/J_f$ is a finitely generated $k[t]$-module equipped with the multiplication structure $t \cdot \bar{r} := f\bar{r}$. 
Since $k[t]$ is a PID, one has 
\begin{equation*}
R/J_f \simeq \oplus_{i=1}^r \frac{k[t]}{(t^{\lambda_i})}.
\end{equation*}
Notice that $\tau_f = \dim_{k} R/(J_f, f)$. Hence, the above isomorphism implies that $$\tau_f = \dim_{k} \left( \left( \oplus_{i=1}^r \frac{k[t]}{(t^{\lambda_i})} \right) / (t) \right) = r.$$
One then has 
\begin{equation}\label{Jordan}
 \mu_f = \sum_{i=1}^r \lambda_i, \quad \tau_f = r,
 \quad e^{BS}(f) = \lambda_1, \quad \text{and} \quad \mu_f - \tau_f = \sum_{i=1}^r (\lambda_i - 1).
 \end{equation}

 One observes that, $\sum_{i=1}^r \lambda_i\leq r\lambda_1$, as well as,  $\lambda_1-1\leq \mu_f-\tau_f$.  Equivalently, one obtains the inequalities
\begin{equation*}
\mu_f\leq e^{BS}(f)\tau_f,\ \text{and}\quad e^{BS}(f)\leq  \mu_f-\tau_f+1. 
\end{equation*}

To the authors knowledge, these inequalities have not been given proper consideration until very recently. For instance, it is proven in \cite[Proposition 1]{Almiron-Meditarianian} that strict inequality $\mu_f<e^{BS}(f)\tau_f$ actually holds.    In order to further address these questions, we introduce the following  containment chain of ideals:
\begin{equation}\label{ratio}
I_f=\langle J_f,f\rangle \subseteq \langle (J_f : f), f \rangle \subseteq  \langle (J_f : f^2), f \rangle \subseteq \cdots \subseteq \langle J_f : f^{e^{\mathrm{BS}}(f)}\rangle=R.
\end{equation}
Define 
\[
\beta_f := \max \{ i\geq 0 \mid  \langle J_f : f^i, f \rangle\subseteq  I_f \}.
\]
Note that $I_f$ is always contained in the left-hand side of the involved inclusion, so the latter is actually an equality.

Given $i\in \{1,\ldots, e^{\mathrm{BS}}(f)\}$, consider the following exact sequence with terms of finite length
$$0\lar R/\langle J_f:f^i \rangle\xrightarrow{f}R/\langle J_f:f^{i-1} \rangle\lar R/\langle J_f:f^{i-1},f\rangle\lar 0.$$  
The appropriate telescopic sum yields
\begin{eqnarray}\label{lengthsum}
	\lambda\left(\frac{R}{J_f}\right)&=&\lambda\left(\frac{R}{I_f}\right)
	+\lambda\left(\frac{R}{\langle J_f:f, f\rangle }\right)
	+\cdots+\lambda\left(\frac{R}{\langle J_f:f^{e^{\mathrm{BS}}(f)-2}, f \rangle}\right)\\ \nonumber
	&+&\lambda\left(\frac{R}{J_f:f^{e^{\mathrm{BS}}(f)-1}}\right).
\end{eqnarray}

 Since $I_f=((J_f:f),f)=\cdots=((J_f:f^{\beta_f}),f)$ by the definition of $\beta_f$, the first $\beta_f+1$ terms of the right side above are one and the same, so the equality can be written as
\begin{equation*}
	\lambda\left(\frac{R}{J_f}\right)=(\beta_f+1)\lambda\left(\frac{R}{I_f}\right)+\sum_{i=1}^{e^{\mathrm{BS}}(f)-\beta_f-1} \lambda\left(\frac{R}{\langle J_f:f^{\beta_f+i}, f \rangle} \right).
\end{equation*} 

The next theorem below establishes related inequalities and, in particular, extends to arbitrary dimension a result of Dimca-Greuel in \cite{DG}.

\begin{Theorem}\label{bounds}  Let $R/\langle f \rangle$ define an isolated hypersurface singularity, where $f\in R=\pot{k}{x_1,\ldots,x_n}$. If $f$ is not quasihomogeneous {\rm (}as stated in {\rm Definition~\ref{quasihomogeneous_redefined}}{\rm )}, then the following inequalities hold$:$
\begin{enumerate}
\item[{\rm (i)}] $\beta_f\leq  e^{\mathrm{BS}}(f)-2$. 
\vspace{0.1cm}
\item[{\rm (ii)}] $\mu_f\leq e^{\mathrm{BS}}(f)\tau_f-(e^{\mathrm{BS}}(f)-\beta_f-1)\leq e^{\mathrm{BS}}(f)\tau_f-1<n\tau_f$.
\vspace{0.1cm}
\item[{\rm (iii)}] $\mu_f\geq (\beta_f+1)\tau_f+2(e^{\mathrm{BS}}(f)-\beta_f-2)+1\geq \tau_f+1$.
\item[{\rm (iv)}] $1< \mu_f/\tau_f<e^{\mathrm{BS}}(f).$
\end{enumerate}
\end{Theorem}
\begin{proof}
(i) As $J_f : f^{e^{\mathrm{BS}}(f)} = \langle 1\rangle =R$ and $R/\langle f \rangle$ is not smooth, it follows that $\beta_f \leq e^{\mathrm{BS}}(f) - 1$.  

Now, assume for the sake of contradiction that $\beta_f > e^{\mathrm{BS}}(f) - 2$.
Then  $\beta_f = e^{\mathrm{BS}}(f) - 1$, hence
 $$I_f = J_f : f^{e^{\mathrm{BS}}(f)-1},$$ 
 by the definition of $\beta_f$.
Here, since $f$ by assumption is  not quasihomogeneous, then $e^{\mathrm{BS}}(f)-1\neq 0$ (see \ref{homog_BSexponent}).

Set  $I' := \langle J_f , f^{e^{\mathrm{BS}}(f)-1} \rangle.$  
The canonical module of $R/I'$ is isomorphic to $\langle J_f : f^{e^{\mathrm{BS}}(f)-1} \rangle / J_f= I_f / J_f$ as $R$-modules. The latter is a cyclic module generated by the residue class of $f$. Consequently, the Artinian ring $R/I'$ is Gorenstein.  

On the other hand, $I'$ is minimally generated by at most $n+1$ elements. By Kunz's theorem (\cite{Kunz}), this implies that $I'$ is a complete intersection. By Theorem~\ref{Tessier},  $f \in \overline{J_f}$. In particular,  
\[
f^{e^{\mathrm{BS}}(f)-1} \in \overline{J_f}.
\]  
Thus, $J_f$ is a reduction of $I'$, hence  $e(J_f,R)=e(I',R)$. Since $\dim(R/J_f)=\dim(R/I')=0$, it follows that $\lambda(R/J_f)=\lambda(R/I')$. Therefore $J_f=I'$. In particular, $f^{e^{\mathrm{BS}}(f)-1}\in J_f$.

This contradicts the definition of $e^{\mathrm{BS}}(f)$, thus completing the proof.

\medskip

(ii) For $i=1,\ldots,e^{\mathrm{BS}}(f)-\beta_f-1$, one has a strict inclusion $I_f\subsetneq \langle J_f:f^{\beta_f+i},f \rangle$. Therefore, $\lambda(R/I_f)-\lambda(R/\langle J_f:f^{\beta_f+i},f \rangle)\geq 1$. Substituting for the main players in the right side of the first claimed inequality in (ii) validates it. The second inequality in (ii) follows from (i), and the third inequality is a consequence of the Brian\c{c}on-Skoda upper bound $e^{\mathrm{BS}}(f)\leq n$.

\medskip

(iii)  The result follows from  (\ref{lengthsum}) and part $(i)$ provided it is shown that, for $i<e^{\mathrm{BS}}(f)-1$, one has
$$\lambda(R/\langle J_f:f^{i}, f \rangle)\geq 2.$$
Assuming otherwise, $\lambda(R/\langle J_f:f^{i}, f\rangle)=1$, i.e., $\langle J_f:f^{i}, f\rangle = \langle x_1,\ldots,x_n\rangle$. Since $f$ is not smooth at the null point, its initial degree as a power series is at least $2$. Therefore,  $\langle x_1,\ldots,x_n\rangle=J_f:f^{i}$. Clearly, then $f\in J_f:f^{i}$, which contradicts the definition of $e^{\mathrm{BS}}(f)$.  

(iv) follows from (ii).

\end{proof}
As a consequence, one rounds up a few cases of small Milnor--Tjurina number difference.
\begin{Corollary}\label{smalldifference} With the same notation and hypotheses of {\rm Theorem~\ref{bounds}}, one has$:$
\begin{itemize}
\item[{\rm (i)}] If $\mu_f-\tau_f\leq 2$ then $e^{\mathrm{BS}}(f)=2$.
\vspace{0.2cm}
\item[{\rm (ii)}]  If $\mu_f-\tau_f\leq e^{\mathrm{BS}}(f)-1$ then $e^{\mathrm{BS}}(f)=2$.
\vspace{0.2cm}
\item[{\rm (iii)}]  If $\mu_f-\tau_f= 3$ then $e^{\mathrm{BS}}(f)\leq 3$ and either  $(\beta_f,\tau_f)=(1,2)$ or else $\beta_f=0$.
\vspace{0.2cm}
\item[{\rm (iv)}] If $\mu_f-\tau_f= 4$ then $e^{\mathrm{BS}}(f)\leq 3$ and either $(\beta_f,e^{\mathrm{BS}}(f),\tau_f)=(1,3,3)$ or $(1,3,2)$, or else $\beta_f=0$.
\end{itemize}
\end{Corollary}
\begin{proof} All parts are consequences of Theorem~\ref{bounds}(iii) and the fact that $\beta_f\leq e^{\mathrm{BS}}(f)-2$. Namely, one has 
$$\mu_f-\tau_f\geq \beta_f\tau_f+2(e^{\mathrm{BS}}(f)-\beta_f-2)+1.$$
As by assumption $f$ is not quasihomogeneous, one has $\tau_f\geq 2$.  Then $\mu_f-\tau_f\geq 2e^{\mathrm{BS}}(f)-3.$ 
\end{proof}

Thus far, no evidence has been found to the effect that the invariant $\beta_f$ is nonzero, i.e., that if $f \in R = \pot{k}{x_1,\ldots,x_n}$ defines an isolated hypersurface singularity, the {\it Euler conductor} $J_f:f$ is not contained in the ideal $I_f$. One asks, more generally:

\begin{Conjecture}\label{conject} {\rm (char $(k)=0$)}
Let $R/\langle f \rangle$ define an isolated hypersurface singularity, with $f \in R = \pot{k}{x_1,\ldots,x_n}$. Then
  $\langle J_f : f^{i+1} \rangle \not\subseteq  \langle (J_f : f^{i}), f \rangle$ for all $0\leq i<e^{BS}(f)$.
\end{Conjecture}  
Note that, by (\ref{homog_BSexponent}), this conjecture trivially holds true if $f$ is quasihomogeneous 

Provided it holds true in general, it has a couple of consequences relating the Milnor and  Tjurina numbers. 
The obtained bounds in all of these cases are better than their counterparts obtained via (\ref{Jordan}).
\begin{Corollary}\label{bounds2} {\rm (char $(k)=0$)}
	 As above, let $R/\langle f \rangle$ define an isolated hypersurface singularity, where $f\in R=\pot{k}{x_1,\ldots,x_n}$ is not quasihomogeneous {\rm (}in the sense of {\rm Definition~\ref{quasihomogeneous_redefined}}{\rm )}.  If {\rm Conjecture~\ref{conject}} holds true, then one has the following inequalities$:$
\begin{enumerate}
	\item[{\rm (i)}] $e^{\mathrm{BS}}(f)\leq\tau_f.$ 
	\item[{\rm (ii)}]  $  \mu_f\leq e^{\mathrm{BS}}(f)\tau_f-\binom{e^{\mathrm{BS}}(f)}{2}=\binom{\tau_f+1}{2}-\binom{\tau_f-e^{\mathrm{BS}}(f)+1}{2}<n\tau_f.$
	
	\vspace{0.2cm}
	
	\item[{\rm (iii)}] 
	$\mu_f\leq\binom{\tau_f+1}{2}$. Moreover, $\mu_f/\tau_f^2\leq 2/3$, and in addition $\mu_f^2\leq  \tau_f^3$ in  the case where $n=3$. 
	\vspace{0.2cm}
\item[{\rm (iv)}]  $ \mu_f\geq\tau_f+2e^{\mathrm{BS}}(f)-3 \geq\tau_f+1.$	
\end{enumerate}
\end{Corollary}

\begin{proof}
(i)  Since  $\lambda({R}/{\langle (J_f:f^{e^{\mathrm{BS}}(f)}),f\rangle})=0$,  one has 
$$\tau_f=\lambda({R}/{\langle J_f,f\rangle})=\sum_{i=0}^{e^{\mathrm{BS}}(f)-1}
\left(\lambda({R}{\langle (J_f:f^i),f\rangle})- \lambda({R}/{\langle (J_f:f^{i+1}),f\rangle})\right)$$
Since the conjecture is assumed to hold true, each summand is at least $1$, thus entailing $\tau_f\geq e^{BS}(f)$.

(ii) By the strictness of the inclusions in
\eqref{ratio}, one has $\lambda(R/\langle J_f:f^{i}, f \rangle)\leq \tau_f-i$ for $ 0\leq i<e^{\mathrm{BS}}(f)$. 
Substituting in (\ref{lengthsum}), yields
\[
\mu_f\leq\sum_{i=0}^{e^{\mathrm{BS}}(f)-1}(\tau_f-i)
=e^{\mathrm{BS}}(f)\tau_f-\binom{e^{\mathrm{BS}}(f)}{2}
=\binom{\tau_f+1}{2}-\binom{\tau_f-e^{\mathrm{BS}}(f)+1}{2}.
\]
In particular, $\mu_f\leq\binom{\tau_f+1}{2}$.
Finally, 
$e^{\mathrm{BS}}(f)\geq2$ since $f$ is not quasihomogeneous (\ref{homog_BSexponent}), and the
Brian\c{c}on--Skoda bound $e^{\mathrm{BS}}(f)\leq n$
yields $\mu_f<n\tau_f$.

(iii) The inequality  $\mu_f\leq\binom{\tau_f+1}{2}$ follows  from  (ii). Dividing both sides of this inequality by $\tau_f^2$ yields  
$$\frac{\mu_f}{\tau_f^2}\leq \frac{1}{2}+\frac{1}{2\tau_f}.$$
 According to Corollary \autoref{leastTau}, $\tau_f\geq 3$, so the above rightmost inequality yields $\mu_f/\tau_f^2\leq1/2+1/6=2/3$.
 For the last assertion of this item, notice that in the case where $n=3$, 
it is shown in \cite{MaZou} that $\mu_f/\tau_f\leq 3/2$. Multiplying this into $\mu_f/\tau_f^2<2/3$ proves that $\mu_f^2\leq \tau_f^3$.

(iv)  Since Conjecture~\ref{conject} is assumed to hold true, then $\beta_f=0$. Therefore, the result follows from \autoref{bounds}(iii).

\end{proof}

\section{An application to logarithmic derivations} \label{logarithmic} 
Throughout this section, the base ring will be a polynomial ring $S:=k[x_1,\ldots,x_n]$, here denoted with $S$ to distinguish from the exhaustive use of the letter $R$ for the ring of formal power series in the previous sections. 

Let $ \Omega^{i}_{S/k}$ ($i\geq 1$) denote the $S$-module of differential $i$-forms.
Recall that $ \Omega^{1}_{S/k}=\Omega_{S/k}$ is $S$-free on generators $dx_1,\cdots,dx_n$ (a property that fails for the ring of formal power series $R=k[[x_1,\cdots,x_n]]$ -- see \cite[Exercise 16.14 ]{EisenbudBook}). Then,  $ \Omega^{i}_{S/k}:=\bigwedge ^i \Omega_{S/k}$.

The module of $k$-linear derivations of $S$ can be thought of as  $\operatorname{Der}_k(S)\simeq \operatorname{Hom}_{S}(\Omega _{S/k},S)$, the $S$-dual of the module of differential $1$-forms (K\"ahler differentials).
Note that $\operatorname{Der}_k(S)$ is freely generated by the partial derivatives $\partial/\partial x_i,\, 1\leq i\leq n$.
Often such derivations are called {\em vector fields}, an algebraic sin one often sticks to.

In close association, one has the $S$-module of \textit{logarithmic derivations} associated to $f\in S$ as  reviewed in the introduction:
$\mathrm{Derlog}_S\langle f \rangle := \{ \eta \in \operatorname{Der}_k(S) \mid \eta(f) \in fS \}$.  

An $S$-submodule of $\Omega_{S/k}$ of interest in the theory of  plane foliations   is 
$$E_{f}:=\{\omega\in \Omega_{k[x_1,x_2]/k}\,|\,df\wedge \omega\in f\,\Omega^{2}_{k[x_1,x_2]/k} \}\quad \textit{\rm (see \cite{Camacho_et_al})}.$$

One can verify that for $n=2$, one has $\mathrm{Derlog}_S\langle f \rangle\simeq E_{f}$ as $S$-modules.

Within $\mathrm{Derlog}_S\langle f \rangle$, there exist naturally arising derivations, which we refer to as \textit{Koszul derivations}.  
These are given by the set  
$$\left\{ f_i \frac{\partial}{\partial x_j} - f_j \frac{\partial}{\partial x_i} \mid 1 \leq i, j \leq n \right\}  
\bigcup \left\{ f \frac{\partial}{\partial x_i} \mid 1 \leq i \leq n \right\}$$
where $f_i = \frac{\partial f}{\partial x_i}$.  
The \textit{module of essential derivations} is then defined as  
$$\mathcal{E}_f: = \mathrm{Derlog}_S\langle f \rangle / \text{Koszul derivations}.$$ 

Let $K({\bf f};S)$ denote the Koszul complex over $S$ of the generating sequence ${\bf f}:=\{f, f_1=\partial f/\partial x_1, \ldots, f_n=\partial f/\partial x_n\}$ of $I_f$.
Let $\mathcal{B}(I_f)$ denote the first boundary of this complex and let $\mathcal{Z}(I_f)$ as before stand for its first cycle (i.e., the first syzygy module of ${\bf f})$.
Then, set $\mathcal{H}(I_f):=\mathcal{Z}(I_f)/\mathcal{B}(I_f)$ for the corresponding first Koszul homology module.

There is a natural map  

\begin{align} \nonumber
	\pi: \mathcal{Z}(I_f) \ (\subseteq S^{n+1}) &\longrightarrow \mathrm{Derlog}_S(f) \\ \nonumber
	(a_0,\ldots,a_n) &\mapsto \sum_{i=1}^{n} a_i\frac{\partial}{\partial x_i}.
\end{align}
This map is clearly an $S$-homomorphism.
Let $\bar{\pi}: \mathcal{Z}(I_f) \to \mathcal{E}_f$ denote the induced composite $S$-homomorphism.

\begin{Lemma}\label{E_fisZ_f} With the above notation, one has$:$
	\begin{enumerate}[label=(\roman*)]
		\item[{\rm (i)}]   $\pi$ is an $S$-isomorphism. 
		\item[{\rm (ii)}] $\bar{\pi}$ induces an $S$-isomorphism $\mathcal{H}(I_f)\simeq \mathcal{E}_f$.
	\end{enumerate}
\end{Lemma}
\begin{proof} (i)
	Injectivity: if $\pi$ vanishes at $(a_0, \dots, a_n)$ then $\sum_{i=1}^{n} a_i \frac{\partial}{\partial x_i} = 0$, hence $a_1 = \dots = a_n = 0$ since the partials are a free basis of  $\operatorname{Der}_k(S)$. On the other hand, since $(a_0, \dots, a_n) \in \mathcal{Z}(I_f)$, then  $a_0 f = 0$, and hence $a_0=0$ as well.
	
	Surjectivity is equally immediate, and left to the reader.

\smallskip
	
	(ii)  By part (i), $\bar{\pi}$ is surjective.
	Next show that the kernel of $\bar{\pi}$ is $\mathcal{B}(I_f)$.
	 Namely, by definition, it is evident that $\mathcal{B}(I_f) \subseteq \ker(\bar{\pi})$. 
	 Now, a vector $(a_0, a_1, \ldots, a_n)\in \mathcal{Z}(I_f)$ belongs to $\ker(\bar{\pi})$ if and only if 
	$$\sum_{i=1}^n a_i \frac{\partial }{\partial x_i} = \sum_{1 \leq i < j \leq n} s_{i,j}\left(f_i \frac{\partial}{\partial x_j} - f_j \frac{\partial}{\partial x_i}\right) + \sum_{i=1}^n t_i\left(f \frac{\partial}{\partial x_i}\right),$$
	for certain $s_{i,j},t_i$ in $S$.
	That is, 
	$$s_{i,j}\left(f_i \frac{\partial}{\partial x_j} - f_j \frac{\partial}{\partial x_i}\right) = s_{ij}\pi(0, \dots, 0, \underset{\substack{\uparrow \\ i}}{-f_j}, 0, \dots, 0, \underset{\substack{\uparrow \\ j}}{f_i}, 0, \dots, 0)$$
	and
	$$t_i\left(f \frac{\partial}{\partial x_i}\right) = t_i\pi(-f_i, 0, \dots, 0, \underset{\substack{\uparrow \\ i}}{f}, 0, \dots, 0).$$
	Since $\pi$ is injective, this implies $\mathcal{B}(I_f) = \ker(\bar{\pi})$ as desired. 
\end{proof}

Determining the minimal number of generators of the module $\mathcal{E}_f$ remains pretty much an open question in the theory of holomorphic foliations, even for a small number of generators. For instance, in \cite{Camacho_et_al}, the authors investigate the number of generators required for vector fields over the affine circle, with a particular focus on the case where the coefficients are rational numbers.
\begin{Example}\rm
	Let $f=x^2+y^2-1\in \QQ[x,y]$. Then $\{ f_x,f_y,f\}=\{2x,2y, x^2+y^2-1\}$.  A calculation with \cite{M2} gives that $\mathcal{Z}(I_f)$ is generated by the columns of the matrix
	$$\left(\begin{matrix}
		-y & x^2-1 & xy\\
		x  & xy &  y^2-1\\
		0  & -2x & -2y
	\end{matrix}\right).
	$$
By Lemma~\ref{E_fisZ_f} (i), one has
$$\mathrm{Derlog}_S(f)=S \left( -y\frac{\partial}{\partial x} + x\frac{\partial}{\partial y}\right) \, + S \left((x^2-1)\frac{\partial}{\partial x}+xy\frac{\partial}{\partial y}\right)\, + S \left(xy\frac{\partial}{\partial x} +  (y^2-1)\frac{\partial}{\partial y}\right).
$$

Trading in the free basis of $\Omega_k(S)$, the $S$-module $E_f$ is generated by the elements
$$\omega_0:=xdx+ydy=(1/2) df,\, \omega_1:=xydx-(x^2-1)dy,\, \omega_2:=(y^2-1)dx-xydy,$$
thus giving $\omega_0\in \langle \omega_1,\, \omega_2\rangle.$
	Therefore, $E_f$ is actually two-generated. 
\end{Example}

\begin{Remark}\rm
	(1) $f$ is up to sign the determinant of the cofactor of the zero entry in the above matrix. In fact, a similar such expression holds  for $f=x_1^2+\cdots +x_n^2-1\in k[x_1,\ldots,x_n]$, for arbitrary $n$ and over any perfect ground field of characteristic $\neq 2$ (see \cite[A propedeutic example]{Aron-Warwick}). It is also a Saito determinantal expression of the free divisor $f$.
	
	(2) The above example goes against the expectation at the end of \cite[Example 2]{Camacho_et_al}.
\end{Remark}
Drawing upon item (ii) of the previous lemma affords a more encompassing formulation of the main result of \cite{Camacho_et_al}.

To see this, given a polynomial $f\in S=k[x_1,\ldots,x_n]$, and a singular point $p\in \mathbb{A}^n_k$ of the affine variety $V(f)\in \mathbb{A}^n_k$, mimicking Definition~\ref{quasihomogeneous_redefined} we will say that 
$V(f)$ has a {\em quasihomogeneous singularity at $p$} if $(J_f)_{\fp}=(I_f)_{\fp}$ in $S_{\fp}$.

\begin{Theorem}\label{E_fprincipal} {\rm (char $(k)=0$)}
Let $f\in S=k[x_1,\ldots,x_n]$ be a polynomial such that, either $V(f)$ is smooth or, else, has a rational isolated singular point $p\in \mathbb{A}^n_k$. Letting $\fp$
denote the defining prime ideal of $p$, one has$:$
	\begin{enumerate}[label=(\roman*)]
		\item[{\rm (i)}] $V(f)$ is smooth if and only if  $\mathcal{E}_f=0$.
		\item[{\rm (ii)}] V$( f )$ has a quasihomogeneous singularity at $p$ if and only if $(\mathcal{E}_{f})_{\fp}$ is a cyclic $S_{\fp}$-module.
	\end{enumerate}
\end{Theorem}
\begin{proof}
	(i)  The ``only if" direction: $V(f)$ is smooth implies that $I_f=(1)$. Moreover,   the Koszul homology module $\mathcal{H}(I_f)$ is annihilated by  the ideal $I_f$. Thus this Koszul homology module must be zero.  Lemma \ref{E_fisZ_f}(2)  then implies that $\mathcal{E}_f=0$.
	
	The ``if" direction: likewise, if by contrary  $I_f$ is a proper ideal of $R$, by the assumption and  Lemma \ref{E_fisZ_f}(2),  the Koszul homology $\mathcal{H}(I_f)$  vanishes. Then, locally at some maximal ideal of $S$, the set $f,f_{x_1},\cdots,f_{x_n}$ is a  regular sequence. But this is absurd as $\dim S=n$.
	\smallskip
	
	(ii) By item $(\rm i)$,  we may assume that $V(f)$ is non-smooth.
	
	The ``if" direction:  assume that $(\mathcal{E}_{f})_{\fp}$ is a cyclic $S_{\fp}$-module. Since $p$ is an isolated singularity, 
	the ideal $(I_f)_{\fp}\subset S_{\fp}$ has codimension $n$, and the ring $(S/I_f)_{\fp}$ is Cohen-Macaulay (Artinian) local ring. 
	Since $\grade(I_f)_{\fp}=n$, the first Koszul homology module of $(I_f)_{\fp}$ is the canonical module of $(S/I_f)_{\fp}$.
	But by assumption and Lemma \ref{E_fisZ_f}(2) the Koszul homology module $\mathcal{H}(I_f)_{\fp}$ is cyclic. Thus, the canonical  module of $(S/I_f)_{\fp}$ is a cyclic module, hence ${(S/I_f)_{\fp}}$ is a Cohen-Macaulay ring of type one. Such a ring must be Gorenstein, see \cite[Chapter 21]{EisenbudBook}. By a result of E. Kunz \cite{Kunz},   $(I_f)_{\fp}$ is forcefully a complete intersection generated by $n$-elements. 
	
	In order to apply Theorem \ref{QuasiHomogeneous}, we need to enlarge the ground ring to $\pot{k}{x_1,\ldots,x_n}$.
	For this we observe that, since $p$ is a rational point, by a change of coordinates we may assume that $\fp=\langle x_1,\ldots,x_n\rangle$.
	Consider the natural inclusions 
	$$k[x_1,\ldots,x_n] \hookrightarrow k[x_1,\ldots,x_n]_{\langle x_1,\ldots,x_n\rangle} \hookrightarrow \pot{k}{x_1,\ldots,x_n}.$$
	Let $I\subset \pot{k}{x_1,\ldots,x_n}$ stand for the extended ideal of $(I_f)_{\fp}$, still  $n$-generated. 
Applying Theorem~\ref{QuasiHomogeneous} -- or rather, the implication (c) $\Rightarrow$ (a) in that theorem --  we are through.
	
	The ``only if" direction:
assume V$(f )$ has a quasihomogeneous singularity at $p$.
By definition, after localization at $\fp$, one has
$f\in \langle f_{x_1},\ldots,f_{x_n}\rangle S_{\fp}$. Hence, over $S_{\fp}$, the Koszul complex degenerates as
\[
K(f,f_{x_1},\ldots,f_{x_n};S_{\fp})
\simeq
K(0,f_{x_1},\ldots,f_{x_n};S_{\fp}).
\]
Thus, taking the respective homology in degree one yields
\[
H_1(f,f_{x_1},\ldots,f_{x_n};S_{\fp})
\simeq
H_0(f_{x_1},\ldots,f_{x_n};S_{\fp})
\simeq
S_{\fp}/(J_f)_{\fp},
\]
a cyclic $S_{\fp}$-module. The result now follows from
Lemma~\ref{E_fisZ_f}.

\end{proof}


\begin{Remark}\rm If one does not care about how quasihomogeneity or other conditions (such as the ones in \cite{Camacho_et_al}) affect the choice of ``nice'' generators, then Lemma~\ref{E_fisZ_f} tells an old story of getting best syzygy generators of an ideal, certainly a common pursue.
	When $n=2$, in the present landscape, $\mathcal{Z}(I_f)$ is a free module everywhere locally on $k[x_1,x_2]$, hence is projective as an $k[x_1,x_2]$-module. But then Seshadri theorem \cite{Sesh} assures that $\mathcal{Z}(I_f)$ is a free module, necessarily of rank two.
	Thus, $E_f$ can always be generated by two elements. Alas, unfortunately, \cite{Sesh} does not tell us how to ``best'' choose such a pair of generators.
\end{Remark}
While the main criterion in Theorem~\ref{E_fprincipal} is of local nature, its next corollary provides a global criterion for quasihomogeneity.

\begin{Corollary}\label{Corolaryfinal} {\rm ($k$ algebraically closed, char$(k)=0$)}
	Let  $f\in S=k[x_1,\ldots,x_n]$ be a polynomial such that $f\in\sqrt{J_f}$. If the variety
	$V(f)\subseteq\mathbb A_k^n$ is singular, with only isolated singularities, then the following are equivalent$:$
	\begin{enumerate}
		\item[\rm (i)] $V(f)$ has a quasihomogeneous singularity at $p$
		for any singular point $p$ in $V(f)$.
		\item[\rm (ii)] $\mathcal E_f$ is a cyclic $S$-module$;$ more precisely, the module $\operatorname{Derlog}_S(f)$ is generated
		by exactly one additional derivation besides the Koszul
		derivations.
		\item[\rm (iii)] $I_f=J_f$.
	\end{enumerate}
\end{Corollary}
	\begin{proof}
		(i) $\Rightarrow$ (ii) Since $V(f)$ has only points as singularities, codim$(I_f)=n$. Hence, any prime ideal $\fp\supseteq I_f$ defines a singular point of $V(f)$. Theorem \ref{E_fprincipal} assures that $(\rm i)$ implies $\nu_{S_{\fp}}(\mathcal{E}_f)_{\fp}\leq 1$ for any prime ideal $\fp\in {\rm Supp}(\mathcal{E}_f)=V(I_f)$. Now apply Forster's Theorem \cite[Theorem 5.7]{Matsumura}, stating that 
		$$\nu(\mathcal{E}_f)\leq \sup\{\nu(\mathcal{E}_f)_{\fp}+\dim(S/\fp)\,|\, \fp \in {\rm Supp}(\mathcal{E}_f)\}=1.$$ 
		Since $V(f)$ is not smooth, again, Theorem \ref{E_fprincipal}(i) implies that $1\leq \nu(\mathcal{E}_f)$. Then $\nu(\mathcal{E}_f)=1$.
	
	(ii) $\Rightarrow$ (iii) Since $J_f\subseteq I_f$, it is enough to show that the equality holds locally at all associated primes of $J_f$. Because $f\in\sqrt{J_f}$, one has $\sqrt{J_f}=\sqrt{I_f}$. Hence $J_f$ has codimension $n$, and since $k$ is algebraically closed, any associated prime $\mathfrak p$ of $J_f$ is of the form
	$\mathfrak p=(x_1-p_1,\ldots,x_n-p_n)$ for some $p_i\in k$. Moreover, $f\in\mathfrak p$, so $\mathfrak p$ defines a singular point of $V(f)$. The cyclicity of $\mathcal E_f$ implies that $(\mathcal E_f)_{\mathfrak p}$ is cyclic.
	Therefore, Theorem  \ref{E_fprincipal}(ii) assures that $(J_f)_{\fp}= (I_f)_{\fp}$ as desired.
	
	(iii) $\Rightarrow$ (i) This follows by localization. 
\end{proof}

\textbf{AI Disclosure.} The theorems and proofs of this paper are all done by human work without resorting to generative AIs.



\begin{thebibliography}{99}

\bibitem{Almiron}
P. Almir\'on,
On the quotient of Milnor and Tjurina numbers for two-dimensional
isolated hypersurface singularities,
Math. Nachr. 295 (2022) 1254--1263.

\bibitem{Almiron-Meditarianian}
P. Almir\'on,
The Tjurina number for Sebastiani--Thom type isolated hypersurface
singularities,
Mediterr. J. Math. 20 (2023) 258.

\bibitem{ABDM}
A.V. Andrade, V. Beorchia, A. Dimca, R.M. Mir\'o-Roig,
Quasi-homogeneous singularities of projective hypersurfaces and
Jacobian syzygies,
arXiv:2502.06290 (2025).

\bibitem{BayerHefez}
V. Bayer, A. Hefez,
Algebroid plane curves whose Milnor and Tjurina numbers differ by
one or two,
Bol. Soc. Bras. Mat. 32 (2001) 63--81.

\bibitem{BS}
J. Brian\c{c}on, H. Skoda,
Sur la cl\^oture int\'egrale d'un id\'eal de germes de fonctions
holomorphes en un point de $\CC^n$,
C. R. Acad. Sci. Paris S\'er. A 278 (1974) 949--951.

\bibitem{Herzog.B}
W. Bruns, J. Herzog,
Cohen--Macaulay Rings,
2nd ed., Cambridge Studies in Advanced Mathematics, vol. 39,
Cambridge University Press, Cambridge, 1998.

\bibitem{Camacho_et_al}
C. Camacho, H. Movasati, C. Hertling,
Algebraic curves and foliations,
Bull. Lond. Math. Soc. 55 (2023) 410--427.

\bibitem{DGPS}
W. Decker, G.-M. Greuel, G. Pfister, H. Sch\"onemann,
{\sc Singular} 4-3-0---A computer algebra system for polynomial
computations,
\url{https://www.singular.uni-kl.de}, 2022.

\bibitem{DG}
A. Dimca, G.-M. Greuel,
On 1-forms on isolated complete intersection curve singularities,
J. Singul. 18 (2018) 114--118.

\bibitem{DimcaSticlaru}
A. Dimca, G. Sticlaru,
Syzygies of Jacobian ideals and weighted homogeneous singularities,
J. Symbolic Comput. 74 (2016) 627--634.

\bibitem{EisenbudBook}
D. Eisenbud,
Commutative Algebra with a View Toward Algebraic Geometry,
Graduate Texts in Mathematics, vol. 150,
Springer-Verlag, New York, 1995.

\bibitem{M2}
D.R. Grayson, M.E. Stillman,
{\sc Macaulay2}, a software system for research in algebraic geometry,
\url{https://macaulay2.com}.

\bibitem{GLS}
G.-M. Greuel, C. Lossen, E. Shustin,
Introduction to Singularities and Deformations,
Springer Monographs in Mathematics,
Springer, Berlin, 2007.

\bibitem{Huneke2}
C. Huneke,
Hilbert functions and symbolic powers,
Michigan Math. J. 34 (1987) 293--318.

\bibitem{HunekeSwanson}
C. Huneke, I. Swanson,
Integral Closure of Ideals, Rings, and Modules,
London Mathematical Society Lecture Note Series, vol. 336,
Cambridge University Press, Cambridge, 2006.

\bibitem{Iarobbino-et-al}
A. Iarrobino, P. Macias Marques, C. McDaniel,
Artinian algebras and Jordan type,
J. Commut. Algebra 14 (2022) 365--414.

\bibitem{Kubota} K. Kubota, On the Hilbert-Samuel Function, Tokyo J. Math.
8 (1985), 439--448.

\bibitem{Kunz}
E. Kunz,
Almost complete intersections are not Gorenstein rings, J. Algebra 28 (1974) 111--115.

\bibitem{L-JT74}
M. Lejeune-Jalabert, B. Teissier,
Cl\^oture int\'egrale des id\'eaux et \'equisingularit\'e,
Ann. Fac. Sci. Toulouse Math. (6) 17 (2008) 781--859.

\bibitem{Liu}
Y. Liu,
Milnor and Tjurina numbers for a hypersurface germ with isolated
singularity,
C. R. Math. Acad. Sci. Paris 356 (2018) 963--966.

\bibitem{MaZou}
H. Ma, H. Zuo,
The quotient of Milnor number by Tjurina number of hypersurface
singularities in arbitrary characteristic,
arXiv:2604.17757v1 [math.AG] (2026).

\bibitem{Marts}
A. Martsinkovsky,
Maximal Cohen--Macaulay modules and the quasihomogeneity of isolated
Cohen--Macaulay singularities,
Proc. Amer. Math. Soc. 112 (1991) 9--18.

\bibitem{Matsumura}
H. Matsumura,
Commutative Ring Theory,
Cambridge Studies in Advanced Mathematics, vol. 8,
Cambridge University Press, Cambridge, 1989.

\bibitem{Milnor}
J. Milnor,
Singular Points of Complex Hypersurfaces,
Annals of Mathematics Studies, vol. 61,
Princeton University Press, Princeton, NJ, 1968.

\bibitem{Miranda}
C.B. Miranda-Neto,
On degrees of logarithmic vector fields of certain schemes,
Arch. Math. (Basel) 119 (2022) 479--488.

\bibitem{Abbas}
A. Nasrollah Nejad,
The Aluffi algebra of hypersurfaces with isolated singularities,
Comm. Algebra 46 (2018) 3553--3562.

\bibitem{Northcott}
D.G. Northcott,
A note on the coefficients of the abstract Hilbert function,
J. Lond. Math. Soc. 35 (1960) 209--214.

\bibitem{Saito}
K. Saito,
Quasihomogene isolierte Singularit\"aten von Hyperfl\"achen,
Invent. Math. 14 (1971) 123--142.

\bibitem{SchejaStorch}
G. Scheja, U. Storch,
\"Uber differentielle Abh\"angigkeit bei Idealen analytischer
Algebren,
Math. Z. 114 (1970) 101--112.

\bibitem{Sesh}
C.S. Seshadri,
Triviality of vector bundles over affine space $K^2$,
Proc. Natl. Acad. Sci. USA 44 (1958) 456--458.

\bibitem{AronSyzgetic}
A. Simis,
Koszul homology and its syzygy-theoretic part,
J. Algebra 55 (1978) 28--42.

\bibitem{Aron-Warwick}
A. Simis,
A quest for irreducible free divisors,
Free Divisors, Warwick, May 2011, slides.

\bibitem{SimisBook}
A. Simis,
Commutative Algebra,
2nd ed., De Gruyter, Berlin--Boston, 2023.


\bibitem{Aron-Wolmer}
A. Simis, W.V. Vasconcelos,
The syzygies of the conormal module,
Amer. J. Math. 103 (1981) 203--224.

\bibitem{Teissier}
B. Teissier,
Cycles \'evanescents, sections planes et conditions de Whitney,
in: Singularit\'es \`a Carg\`ese,
Ast\'erisque 7--8,
Soci\'et\'e Math\'ematique de France, Paris, 1973,
pp. 285--362.

\bibitem{Watari}
M. Watari,
Plane curve singularities whose Milnor and Tjurina numbers differ
by three,
Adv. Stud. Pure Math. 46 (2007) 273--298.


\bibitem{ZhangZuo} Y. Zhang and H. Zuo, The Hassanzadeh-Nasrollah Nejad-Simis Conjecture on Euler Conductors, arXiv:2610.09833v1 [math.AC]

\end{thebibliography}
\end{document}